\documentclass[12pt,leqno]{article}
\usepackage{amsmath, amsthm, amsfonts, amssymb, color}
\usepackage{mathrsfs}
\newtheorem{thm}{Theorem}[section]
\newtheorem{cor}[thm]{Corollary}
\newtheorem{lem}[thm]{Lemma}

\theoremstyle{definition}

\newcommand{\scr}[1]{\mathscr #1}
\definecolor{wco}{rgb}{0.5,0.2,0.3}

\numberwithin{equation}{section} \theoremstyle{remark}

\newcommand{\ua}{\uparrow}

\title{{\bf Log-Harnack  Inequality for Stochastic Differential Equations in Hilbert Spaces and its Consequences}\footnote{Supported in
 part by WIMCS, NNSFC(10721091) and the 973-Project.}
}
\author{
{\bf Michael R\"ockner$^{b),c)}$ and Fneg-Yu Wang$^{a),d)}$\footnote{Corresponding author.
wangfy@bnu.edu.cn; F.Y.Wang@swansea.ac.uk}}\\
\footnotesize{$^{a)}$School of Math. Sci. \& Lab. Math. Com. Sys.,
Beijing Normal
University, Beijing 100875, China}\\
\footnotesize{$^{b)}$Department of Mathematics, Bielefeld
University, D-33501 Bielefeld, Germany}\\
\footnotesize{$^{c)}$Department of Mathematics and Statistics, Purdue University, W. Lafayette, IN 47907, USA}\\
\footnotesize{$^{d)}$Department of Mathematics, Swansea University,
Singleton Park, SA2 8PP, UK} }
\begin{document}
\def\R{\mathbb R}  \def\ff{\frac} \def\ss{\sqrt} \def\BB{\mathbb
B}
\def\N{\mathbb N} \def\kk{\kappa} \def\m{{\bf m}}
\def\dd{\delta} \def\DD{\Delta} \def\vv{\varepsilon} \def\rr{\rho}
\def\<{\langle} \def\>{\rangle} \def\GG{\Gamma} \def\gg{\gamma}
  \def\nn{\nabla} \def\pp{\partial} \def\tt{\tilde}
\def\d{\text{\rm{d}}} \def\bb{\beta} \def\aa{\alpha} \def\D{\scr D}
\def\EE{\mathbb E} \def\si{\sigma} \def\ess{\text{\rm{ess}}}
\def\Ric{\text{\rm{Ric}}} \def\Hess{\text{\rm{Hess}}}\def\B{\scr B}
\def\e{\text{\rm{e}}} \def\ua{\underline a} \def\OO{\Omega}  \def\oo{\omega}     \def\tt{\tilde} \def\Ric{\text{\rm{Ric}}}
\def\cut{\text{\rm{cut}}} \def\P{\mathbb P} \def\ifn{I_n(f^{\bigotimes n})}
\def\C{\scr C}      \def\aaa{\mathbf{r}}     \def\r{r}
\def\gap{\text{\rm{gap}}} \def\prr{\pi_{{\bf m},\varrho}}  \def\r{\mathbf r}
\def\Z{\mathbb Z} \def\vrr{\varrho} \def\ll{\lambda}
\def\L{\scr L}\def\Tt{\tt} \def\TT{\tt}\def\II{\mathbb I}
\def\i{{\rm i}}\def\Sect{{\rm Sect}}\def\E{\mathbb E} \def\H{\mathbb H}

\maketitle
\begin{abstract} A logarithmic type Harnack inequality  is established for the semigroup
of solutions to a stochastic differential equation in Hilbert spaces
with non-additive noise. As applications, the strong Feller property
as well as the entropy-cost inequality for the semigroup are derived
with respect to the corresponding distance (cost function).
\end{abstract} \noindent
 \vskip 2cm

\section{Introduction} Under a curvature condition the second named author established
the following type dimension-free Harnack inequality for diffusion semigroups on a Riemannian manifold $M$
(\cite{W97}):

$$(P_t f)^\aa (x)\le (P_t f^\aa)(x)\e^{c(t) \rr(x,y)^2},\ \ \ f\ge 0, t>0, \, \alpha >1,\, x,y\in M,$$
where $c(t)>0$ is explicitly determined by $\alpha$ and
the curvature lower
bound. This inequality has been efficiently applied to the study of
functional inequalities for the associated Dirichlet form, the
 hyper-/super-/ultracontractivity properties of the semigroup,
strong Feller property as well as estimates on the heat kernel of
the semigroup (cf. \cite{GW02, RW03, Wbook, DRW08} and references
therein). To establish this inequality for diffusions with curvature
unbounded below, a coupling method is developed in \cite{ATW}. This
method works also for infinite dimensional SPDE provided the noise
is additive and non-degenrate, see e.g. \cite{W07, Liu, LW, OY,
DRW08}  for Harnack inequalities for several different classes of SPDE. The
aim of this paper is to extend the study to stochastic differential
equations with non-additive noises.

Let us start from the following It\^o stochastic differential
equation on $\R^n$:

$$\d X_t = b(X_t)\d t+\si(X_t)\d B_t,$$
where $b: \R^n\to \R^n$ and $\si: \R^n\to \R^n\otimes \R^n$ are
continuous,  $B_t$ is the Brownian motion in $\R^n$. In this case
the solution is a diffusion process with corresponding generator
(Kolmogorov operator)

$$L= \sum_{i,j=1}^n a_{ij}(x) \pp_i\pp_j + \sum_{i=1}^n b_i(x) \pp_i,$$
where $(a_{ij})_{1\le i,j\le n}=\ff 1 2 \si^*\si$. If $a_{ij}$ and $b_i$ are regular enough
such that Bakry-Emery's $\GG_2$ condition (see \cite{BE})

\begin{equation}\label{G2}\GG_2(f,f):= \ff 1 2  L\<a\nn f,\nn f\> - \<a\nn f,
\nn Lf\>\ge -K \<a\nn f,\nn f\>\end{equation}
holds for all smooth
$f$ and some constant $K$, then the curvature condition used in
\cite{W97} holds for the Riemannian metric $\<u, v\>_a:= \<a^{-1} u,
v\>, u,v\in\R^n$, induced by the diffusion coefficient. Thus, one
derives the desired Harnack inequality for the associated diffusion
semigroup. Theoretically one may use this argument to
establish the Harnack inequality for non-constant $a$ also in
infinite dimensions. For $n \to \infty$
condition (\ref{G2})  is, however, too complicated to
verify or does not hold. This is the main reason why all existing
results in this direction for  infinite dimensional SDE are
merely proved for additive noise (i.e. for the constant diffusion case).

 In this paper we shall analyse the following log-Harnack inequality allowing
 the diffusion to be non-constant:
 $$P_t\log f(x)\le  \log P_t f(y) + \ff{K\rr_a(x,y)^2}{2(1-\e^{-2Kt})},\ \ \ t>0,
 x,y\in \R^n, f>0,$$
 where $\rr_a$ is the distance induced by the metric $\<\cdot,\cdot\>_a.$
 This inequality was first presented in the proof of \cite[Lemma 4.2]{BGL} under
 the $\GG_2$ condition (\ref{G2}) for $P_t f$ in place of $f$, which
 is crucial for the proof of the HWI inequality
 \cite[Theorem 4.3]{BGL}.
We will see  in
 Section 2  that this type of inequality can be derived by using a standard dissipative type condition
 which is explicit and dimension free.  Combining  this observation with an approximation argument,
 we are able to establish the inequality for infinite dimensional diffusions on Hilbert spaces, and furthermore
 derive the strong Feller property of the semigroup and entropy inequalities for the heat kernel.

 We will work with the following semi-linear stochastic differential equation
 on a separable Hilbert space $(\H, \<\cdot, \cdot\>, \|\cdot\|)$ (cf. \cite{DZ}):

\begin{equation}\label{1.2} \d X_t= \big(AX_t +F(X_t)\big)\d t + \si(X_t)\d
W_t,\end{equation} where $W_t$ is a cylindrical Brownian motion on
$\H$ on some filtered probability space $(\Omega, \mathcal F, \mathbb P, (\mathcal F_t))$.
$F$ is a Lipschtiz continuous function on $\H$, and $\si(x)=
\tilde \si_1(x)+ \tt\si_0$ for a linear operator $\tt\si_0$ and a
Hilbert-Schmidt operator-valued
 function $\tt\si_1$ such that $\si^*\si\ge \tt\si_0^2$. We shall assume:

 \begin{enumerate}
 \item[$(H1)$]  $A$ is a self-adjoint operator on $\mathbb H$
 generating a contractive compact semigroup
 $T_t$. In this case $-A$ has discrete spectrum $0\le  \ll_1\le \ll_2\le\cdots$ with
 corresponding eigenbasis $\{e_i\}_{i\ge 1}$ of $\H$.
 Let $\H_n={\rm span}\{e_1,\cdots, e_n\}, n\ge 1.$
 \item[ $(H2)$] \ $\tt\si_0 e_i=q_i e_i$ for a sequence $\{q_i>0:\ i\ge 1\}$ such that $\si^*\si\ge \tt\si_0^2$ and
 $\sum_{i=1}^\infty \ff{q_i^2}{1+\ll_i}<\infty.$
 \item[$(H3)$] $\tt\si_1:= \si-\tt\si_0$ is Hilbert-Schmidt and there exists a constant $C>0$ such that
 $$\|F(x)-F(y)\|+\|\tt\si_1(x)-\tt\si_1(y)\|_{HS}\le C\|x-y\|,\ \ \ x,y\in \H.$$ \item[ $(H4)$] There exists a
 constant $K\in \R$ such that
 $$2\<F(x)-F(y), \tt\si_0^{-2}(x-y)\>
  +\|\tt\si_0^{-1} (\tt\si_1(x)-\tt\si_1(y))\|_{HS}^2\le K\|\tt\si_0^{-1}(x-y)\|^2$$ holds for all
  $x,y\in \H$ with $x-y\in  \cup_{n=1}^\infty \H_n.$ \end{enumerate}

We note that $(H3)$ implies $(H4)$ in case $\tilde \si_0$ and $\tilde \si _0^{-1}$
are both bounded.
Obviously, $(H1)$--$(H3)$ imply the existence and the uniqueness of the mild solution to
(\ref{1.2}), that is, for any $x\in \H$ there exists a unique $\H$
valued adapted process $X_t$, which is continuous in $L^2 (\Omega, \mathbb P)$,
such that (cf. \cite{DZ})

$$X_t= T_t x+ \int_0^t T_{t-s} F(X_s)\d s+ \int_0^t T_{t-s} \si(X_s)\d W_s.$$ Let $P_t$ be the associated
Markov semigroup, i.e.

$$P_t f(x)= \E f(X_t),\ \ \ f\in \B_b(\H),$$ where $\B_b(\H)$ is the set of all bounded
measurable functions on $\H$.
In this paper  we shall establish a log-Harnack inequality for $P_t$ by using $(H4)$
in place of the $\GG_2$ condition.

\begin{thm}\label{T1.1} If $(H1)$-- $(H4)$ hold then for any strictly
positive $f\in \B_b(\H)$,

$$P_t \log f(x) \le \log P_t f(y)+ \ff{K\|\tt\si_0^{-1}(x-y)\|^2}{2(1-\e^{-Kt})},\ \ \
t>0, x,y\in\H,$$ where $\|\tt\si_0^{-1}x\|^2:= \sum_{i=1}^\infty
q_i^{-2} \<x,e_i\>^2\in [0,\infty].$  \end{thm}

As applications of Theorem \ref{T1.1}, we have the following results on the strong Feller
property, heat kernel inequality and entropy-cost inequality.  To state these results, let us
introduce some notions. Let

$$\H_{0}=\{x\in \H: \|x\|_{0}:= \|\tt\si_0^{-1} x\|<\infty\}.$$
We call $P_t$ $\H_0$-strong Feller if for any $f\in \B_b(\H)$,

\begin{equation}\label{2.3a}
   \lim_{\|y-x\|_0\to 0 } P_t f(y)= P_t f(x),\ \ \ x\in\H.
\end{equation}

When $\tt\si_0^{-1}$ is bounded then $\H_0=\H$
and $\mathbb H_0$-strong Feller, implies $\mathbb H$-strong Feller. 
Next, let $P_t$ be $\H_0$-strong Feller
and let $\mu$ be a probability measure on $\mathbb H$ such that
for some $C$, $\alpha> 0$,
\begin{equation}\label{1.3a}
 \int P_t f \d\mu \leq C e^{\alpha t}   \int f \d \mu\quad \text{for all } f \in \B_b(\mathbb H), \; f \geq 0,
\end{equation}
(which holds e.g. if $\mu $ is $P_t$-invariant). Such measures always exist. Take e.g.
for $x_0 \in \mathbb H$, $\mu(d y):= \int _0^\infty e^{-s}P_s (x_0, dy)ds$.
Then $\mu$ satisfies \eqref{1.3a} with $\alpha = 1 = C$. Suppose that
$\mu$ is fully supported on $\H_0$, i.e. $\mu(U)>0$ for every nonempty
$\|\cdot\|_0$-open set $U \subset \mathbb H_0$.
  Then it is easy to see that for every $x \in \H_0$, 
  $P_t(x,dy)$ has a transition density $p_t(x,y)$ with respect to
 $\mu$.
 
\textbf{Remark.}  
 Obviously $(\mathbb H_0, \|\cdot \|_0)$ is separable. Hence there exists $\mu$ as in
 \eqref{1.3a} fully supported on $\mathbb H_0$. Indeed, take a countable
 $\|\cdot\|_0$-dense subset $\{x _n| n \in \N\}$ of $\mathbb H_0$.
 Then
 \[\mu(dy) := \sum_{n=1}^\infty \frac{1}{2^n} \int_0^\infty e^{-s}P_s (x_n, dy) \d s\]
 is a probability measure on $\mathbb H$, satisfying \eqref{1.3a}. Furthermore,
 if $U \subset \mathbb H_0$ is $\|\cdot\|_0$-open such that $\mu(U)=0$.
 Then for
 $\varphi(x) := \inf \{\|x - y\|_0 :\  y \in U^c\}$
 we have
 \[\int \varphi \d\mu =0.\]
 Hence by a diagonal argument we can find a zero sequence $(t_k)_{k\in \N}$
 such that $P_{t_k}\varphi (x_n) =0$ for all $k, n \in \N$. Taking
 $k \rightarrow \infty$ we obtain $\varphi(x_n) =0$ for all $n \in \N$.
 But if $U \neq \emptyset$, then $x_{n_0}\in U$ for some $n_0 \in \N$,
 so $\varphi (x_{n_0})>0$. This contradiction shows that $U= \emptyset$.\medskip

 Finally, for two probability measures $\mu_1,\mu_2$ on $\H$, let
 $W_0(\mu_1,\mu_2)$ be the $L^2$-Wasserstein distance or $L^2$-transportation cost between  them with respect to
 the cost function $(x,y)\mapsto \|x-y\|_0.$ More precisely, with
 $\C(\mu_1,\mu_2)$ denoting the set of all couplings of $\mu_1$ and $\mu_2$, we have

 $$W_0(\mu_1,\mu_2)^2 =\inf_{\pi\in\C(\mu_1,\mu_2)} \int_{\H\times\H} \|(\tt\si_0^{-1}
 (x-y)\|^2\pi(\d x,\d y).$$

 \begin{cor}\label{C1.2} Let $(H1)$-- $(H4)$ hold. Then: \begin{enumerate}
 \item[$(1)$]  For any $t>0$, $P_t$ is $\H_0$-strong Feller.
 Let $\mu$ be $P_t$-subinvariant $($i.e., \eqref{1.3a} holds
 with $C = 1, \; \alpha =0)$. Then \eqref{2.3a} holds for all $\mu$-exponentially
 integrable functions $f$.
 \item[ $(2)$]
 Let $\mu$ be as in \eqref{1.3a} above, fully supported on $\mathbb H_0$.
 Then for every $x \in \mathbb H_0$, $P_t(x, dy)$ has a transition density
 $p_t(x,y)$ satisfying the following entropy inequality
  $$\int_\H p_t(x,z)\log p_t(x,z) \mu(\d z) \le \log C+ \alpha t -\log \int_\H \exp
 \Big[-\ff{K\|x-y\|_0^2}{2(1-\e^{-Kt})}\Big]\mu(\d y),\ \ \ t>0, x\in \H.$$
  \item[$(3)$] Let $\mu$ be $P_t$-subinvariant. Then
  the following entropy-cost inequality holds for the adjoint
  operator $P_t^*$ of $P_t$
  in $L^2(\mu)$:
  $$\mu((P_t^*f)\log P_t^* f)\le \ff{K}{2(1-\e^{-Kt})} W_0(f\mu,\mu)^2,\ \ t>0,
   f\ge 0, \mu(f)=1.$$
   \end{enumerate} \end{cor}

   In Section 2 we shall prove the log-Harnack inequality for diffusion semigroups on $\R^n$ and
   then we extend this to an infinite dimensional setting in Section 3 by finite-dimensional approximations.
    Finally, Corollary \ref{C1.2} will be proved in Section 4.

   \section{Log-Harnack inequality on $\R^n$}
   Consider the following SDE on $\R^n$:

   \begin{equation}\label{2.1} \d X_t= b(X_t)\d t+\si(X_t)\d B_t,\end{equation} where $B_t$ is Brownian motion on $\R^n$,
   $b: \R^n\to \R^n$ and $\si: \R^n\to \R^n\otimes \R^n$ are locally Lipschitzian
   and of at most linear growth. Hence the equation has a unique
   strong solution, which is non-explosive.
   Let $\tt\si_0$ be a (strictly) positive definite
   symmetric matrix such that $\si^*\si\ge \tt\si_0^2$. Assume that

   \begin{equation}\label{2.2} \|(\tt\si_0^{-1} (\si(x)-\si(y)\|_{HS}^2+ 2\<\tt\si_0^{-1}
   (b(x)-b(y)), \tt\si_0^{-1}(x-y)\> \le K \|\tt\si_0^{-1}(x-y)\|^2,\ \ \ x,y\in \R^n\end{equation} holds for
   some constant $K\in \R$.

   \begin{thm}\label{T2.1} Assume $(\ref{2.2})$ and that the solution to
   $(\ref{2.1})$ is non-explosive. Then the associated Markov semigroup $P_t$ satisfies

   $$P_t \log f(x) \le \log P_t f(y)+ \ff{K\|\tt\si_0^{-1}(x-y)\|^2}{2(1-\e^{-Kt})},\ \ \
t>0, x,y\in\R^n$$  for all $f\in \B_b(\R^n),$ $f\geq 0$. \end{thm}
By a standard approximation argument we can assume that $f \in C_b^\infty(\R^n)$.
Furthermore, we can approximate $b,\sigma$ by smooth $b_n, \sigma_n$ such that
the corresponding semigroups converge pointwise on $f\in C_b^\infty(\R^n)$ and
such that $P_tC_b^2 \subset C_b^2$ for all $t > 0$. \bigskip

To prove the log-Harnack inequality, we need the following gradient estimate on $P_t$.

\begin{lem}\label{L2.2} Under the assumptions of Theorem $\ref{T2.1}$ we have 

$$\|\tt\si_0 \nn P_t f\|^2(x)\le \e^{Kt} P_t \|\tt\si_0 \nn f\|^2(x),\ \ \ f\in C_b^1(\R^n), x\in\R^n.$$
\end{lem}

\begin{proof} For $x,y\in \R^n$, let $X_t$ and $Y_t$ be the solutions
to (\ref{2.1}) with $X_0=x$ and $Y_0=y$ respectively. By It\^o's
formula and (\ref{2.2}) we obtain

\begin{equation*}\begin{split} &\d \|\tt\si_0^{-1} (X_t-Y_t)\|^2 =
2\<\tt\si_0^{-1} (X_t-Y_t),
\tt\si_0^{-1} (\si(X_t)-\si(Y_t))\d B_t\> \\
&+ \big\{\|(\tt\si_0^{-1} (\si(X_t)-\si(Y_t))\|_{HS}^2+
2\<\tt\si_0^{-1}
   (b(X_t)-b(Y_t)), \tt\si_0^{-1}(X_t-Y_t)\>\big\}\d t\\
   &\le 2\<\tt\si_0^{-1} (X_t-Y_t),
\tt\si_0^{-1} (\si(X_t)-\si(Y_t))\d B_t\> +
K\|\tt\si_0^{-1}(X_t-Y_t)\|^2\d t.\end{split}\end{equation*} Since
the solution to (\ref{2.1}) is non-explosive, this implies

$$\E\|\tt\si_0^{-1}(X_t-Y_t)\|^2\le \e^{Kt} \|\tt\si_0^{-1}(x-y)\|^2.$$ Therefore,

\begin{equation*}\begin{split} \|\tt\si_0 \nn P_t f\|^2(x)
&=\limsup_{y\to x} \ff{|P_t f(y)-P_t f(x)|^2}{\|\tt\si_0^{-1}
(x-y)\|^2} = \limsup_{y\to x}\Big( \ff{\E(f(Y_t)-f(X_t))}{\|\tt\si_0^{-1}(x-y)\|}\Big)^2\\
&\le \limsup_{y\to x}
\Big(\E\ff{|f(Y_t)-f(X_t)|^2}{\|\tt\si_0^{-1}(Y_t-X_t)\|^2}\Big)
\ff{\E \|\tt\si_0^{-1} (X_t-Y_t)\|^2}{\|\tt\si_0^{-1}(x-y)\|^2}\\
& \le \e^{Kt} \E \|\tt\si_0\nn f\|^2(X_t).
\end{split}\end{equation*}
This implies  the desired gradient estimate. \end{proof}

\ \newline\emph{Proof of Theorem \ref{T2.1}.} We may assume $ f\ge 1.$
For fixed $x\in \R^n$, let $X_0=x.$ By It\^o's formula we have

\begin{equation*}\begin{split} \d\log P_{t-s}f(X_s) &= \<\nn\log P_{t-s}
f(X_s), \si(X_s)\d B_s\> + L\log P_{t-s} f(X_s) \d s -
\ff{LP_{t-s}f}{P_{t-s}f}(X_s)\d s\\
&= \<\nn\log P_{t-s} f(X_s), \si(X_s)\d B_s\> - \ff 1 2 \|\si\nn\log
P_{t-s}f\|^2(X_s)\d s.\end{split}\end{equation*} Letting

$$\tau_k=\inf\{t\ge 0:\ \|X_t\|\ge k\},\ \ \ k\ge 1,$$ we obtain

$$\E\log P_{t-s\land\tau_k}f(X_{s\land\tau_k})- P_t f(x) = -\ff 1 2
\E\int_0^{s\land \tau_k} \|\si \nn \log P_{t-r}f\|^2(X_r)\d r.$$
Since the process is non-explosive, we have $\tau_k\to\infty$. Thus,
due to the dominated convergence theorem, as $k\to\infty$ the
left-hand side goes to $P_s \log P_{t-s}f(x) - \log P_tf(x)$, while
by the monotone convergence theorem, the right-hand side goes to
$-\ff 1 2 \int_0^s P_r\|\si \nabla \log P_{t-r} f\|^2(x)\d r.$ So,
$\int_0^tP_r\|\si\nabla \log P_{t-r}f\|^2(x)\d r<\infty$  and

\begin{equation}\label{DD}
P_s\log P_{t-s}f(x) -\log P_tf(x)= -\ff 1 2 \int_0^s P_r\|\si\nn\log
P_{t-r}f\|^2(x)\d r,\ \ s\in [0,t].
\end{equation}
Now, for fixed $x,y\in \R^n, t>0$, let

$$x_s= (x-y) h_s+y,\ \ \ \ s\in [0,t],$$
 where $h\in C^1([0,t],\R)$ such that $h_0=0$ and $h_t=1.$
 By Lemma \ref{L2.2}, (\ref{DD}) and noting that $\si^*\si\ge \tt\si_0^2$,
 we have, since $s\mapsto P_s \log P_{t-s} f(x_s)$
 is absolutely continuous by Lemma \ref{L2.2}, that

 \begin{equation*}
 \begin{split}
 &P_t \log f (x) - \log P_t f (y) \\
 =&\int _0^t\ff{\d}{\d s} (P_s \log P_{t-s} f)(x_s)\d s\\
 =& -\ff 1 2 \int _0^t\left\{P_s \|\si\nn \log P_{t-s} f\|^2(x_s)
 +h_s' \<x-y, \nn P_s\log P_{t-s}f\>(x_s)\right\}\d s \\
 \le &-\ff 1 2 \int _0^t\left\{\e^{-Ks} \|\tt\si_0\nn P_s\log P_{t-s}f\|^2(x_s)
 +|h_s'|\cdot\|\tt\si_0^{-1}(x-y)\|\cdot\|\tt\si_0\nn P_s\log P_{t-s} f\|\right\}\d s\\
 \le& \ff {\|\tt\si_0^{-1}(x-y)\|^2} 2 \int _0^t\e^{Ks} |h_s'|^2  \d s.\end{split}\end{equation*}
 Letting

 $$h_s= \ff{1-\e^{-Ks}}{1-\e^{-Kt}},\ \ \ s\in [0,t],$$
 we complete the proof. \qed

   \section{Proof of Theorem \ref{T1.1}}

 For any $n\ge 1,$  let $\pi_n:
   \H\to \H_n:=\text{span}\{e_1,\cdots, e_n\}$  be the orthogonal projection.
    Let $W_t^n=\pi_n W_t, A_n= \pi_n A, \si_n=\pi_n \si, \si_{i,n}
   = \pi_n \tt\si_i (i=0,1),$ and $F_n= \pi_n F.$ By $(H1)$ and $(H2)$ we have

   \begin{equation}\label{3.1} A_n x=Ax,\ \ \si_{0,n} x= \tt\si_0 x,\ \ \ x\in \H_n.\end{equation} Consider
   the following SDE on $\H_n$:

   $$\d X_t^n =( A_n X_t^n +F_n(X_t^n))\d t +\si_n(X_t^n)\d W_t^n,\ \ \ X_0^n=\pi_n X_0.$$By $(H3)$
   we see that both $b_n(x):= A_n x+F_n(x)$ and $\si_n(x)$ are
   Lipschitzian
   in $x\in \H_n$. So, this equation has a unique solution.
   Let $P_t^n$ be the associated Markov semigroup.  Moreover, by $(H4)$,
   since $A\le 0$ and by (\ref{3.1}) we have

\begin{equation*}\begin{split} & 2\<\si_{0,n}^{-1} (b_n(x)-b_n(y)), \si_{0,n}^{-1}
(x-y)\> +\|\si_{0,n}^{-1} (\si_n(x)-\si_n(y))\|_{HS}^2\\
&\le 2\<\tt\si_0^{-1} (F(x)-F(y)), \tt\si_0^{-1}(x-y)\> + \|\tt\si_0^{-1}(\si(x)-\si(y))\|_{HS}^2\\
&\le K\|\tt\si_0^{-1}(x-y)\|^2,\ \ \ x,y\in
\H_n.\end{split}\end{equation*} Thus, Theorem \ref{T2.1} implies
that for $f\in C_b(\mathbb H_n)$

\begin{equation}\label{3.2}  P_t^n \log f(x) \le \log P_t^n f(y)+
\ff{K\|\tt\si_0^{-1}(x-y)\|^2}{2(1-\e^{-Kt})},\ \ \ t>0,
x,y\in\H^n.\end{equation} So, to derive the   inequality for $P_t$,
we need only to prove that

\begin{equation}\label{3.3} \lim_{n\to\infty} \E\|X_t^n-X_t\|^2=0,\ \ \ X_0=x\in
\bigcup _{n\geq 1}\H_n.\end{equation}
Indeed, this implies that for any Lipschitzian
function $f$ on $\H$, such that $f = f\circ \pi_m$ for some $m\in\N$

$$\lim_{n\to \infty} |P_tf(x)-P_t^nf(\pi_n x)|\le \|f\|_{Lip}\lim_{n\to\infty}
 \E\|X_t-X_t^n\|=0,\quad \forall x \in \bigcup_{n \geq 1}\mathbb H_n.$$
 Therefore, by letting $n\to\infty$ in (\ref{3.2}) we derive the desired
 log-Harnack inequality for such
 Lipschitzian functions first for $x \in \bigcup_{n\geq 1}\H _n$,
 but then since this set is dense in $\H$ and $P_t f$ is
 continuous for all such $f$, hence also for $\log (f+ \varepsilon )$,
 we obtain it for all $x \in \H$. Finally, we extend it
 for all $f\in \B_b(\H)$ by the monotone class theorem.

 In order to prove (\ref{3.3}), let

 $$Y_t= \int_0^t T_{t-s}\tt\si_0\d W_s,\ \ Y_t^n =\pi_n Y_t= \int_0^t \e^{(t-s)A_n} \si_{0,n}\d W_s^n,\ \ \ t>0.$$
 By $(H2)$ we have

 $$\sup_{s\in [0,t]}\E\|Y_s\|^2<\infty,$$ so that the dominated convergence theorem
 implies

 \begin{equation}\label{3.4}\lim_{n\to\infty}\E\|Y_t-Y_t^n\|^2=0,\ \ \  \lim_{n\to\infty} \int_0^t \E\|Y_s-Y_s^n\|^2\d s =0.
 \end{equation} Let $$Z_t= X_t-Y_t, \ \ \ \ Z_t^n= X_t^n-Y_t^n.$$
 By (\ref{3.4}) it suffices to prove

 \begin{equation}\label{3.5} \lim_{n\to\infty} \E\|Z_t-Z_t^n\|^2=0.\end{equation}
We have

 \begin{align}
 \label{3.6a}&\d Z_t= (AZ_t +F(Z_t+Y_t))\d t + \si_1(Z_t+Y_t)\d W_t,\\
 \label{3.7a}
 &d Z_t^n= (A_nZ_t^n +F_n(Z_t^n+Y_t^n))\d t + \si_{1,n}(Z_t^n+Y_t^n)\d W_t^n.
  \end{align}
 To be precise, \eqref{3.6a} is first meant in the mild sense. But by our
 assumptions it also has a unique variational solution (see e.g. \cite{10a}).
 Since both are analytically weak solutions and these are unique (see e.g.
 the recent paper \cite{2a}, where uniquenss of analytically weak solutions
 is proved for an even more general class of equations),
 we see that $Z_t$ defined above solves \eqref{3.6a} in the variational sense,
 so that It\^o's formula applies to $\|Z_t - Z_t^n\|^2$. Due to
 \eqref{3.1} we have
\begin{equation*}\begin{split}  \d(Z_t-Z_t^n) =
&\Big(A(Z_t-Z_t^n)+F(X_t)-F_n(X_t^n)\Big)\d t\\
&+ (\si_1(X_t)-\pi_n \si_1(X_t^n))\d W_t^n + \si_1(X_t) \d
(W_t-W_t^n).\end{split}\end{equation*}
 So, by I\^o's formula and $(H3)$ we obtain

 \begin{equation}\label{3.6}\begin{split}& \d \|Z_t-Z_t^n\|^2 \le C_1 \Big(\|F(X_t)-F_n(X_t^n)\|\cdot\|X_t-X_t^n\|\\
 &\qquad\qquad\qquad\qquad\qquad +
 \|\si_1(X_t)- \pi_n \si_1(X_t^n)\|_{HS}^2 +\sum_{i>n} \|\si_1(X_t)e_i\|^2\Big)\d t\\
  &\le C_2\Big(\|Z_t-Z_t^n\|^2 +\|Y_t-Y_t^n\|^2 +\|(1-\pi_n)F(X_t)\|^2+\sum_{i>n}
 \|\si_1(X_t)e_i\|^2\Big)\d t\end{split}\end{equation} for some constants $C_1, C_2>0.$
 Since by $(H3)$ we have

 $$\sup_{s\in [0,t]} \E \Big(\|F(X_s)\|^2 +\|\si_1(X_s)\|_{HS}^2\Big)\le C_3
  \sup_{s\in [0,t]}(1+\E\|X_s\|^2)<\infty,$$ by the dominated convergence theorem

  $$\vv_n:= C_2\E \int_0^t\Big(\|Y_s-Y_s^n\|^2 +\|(1-\pi_n)F(X_s)\|^2+\sum_{i>n}
 \|\si_1(X_s)e_i\|^2\Big)\d s\to 0$$ as $n\to\infty.$  Thus, it follows from (\ref{3.6}) that

 $$\lim_{n\to\infty} \E\e^{-C_2 t} \|Z_t-Z_t^n\|^2 \le \lim_{n\to\infty} \vv_n =0.$$
 Therefore, (\ref{3.5}) holds.
\hfill$\Box$

   \section{Proof of Corollary \ref{C1.2}}

 It is sufficient to prove \eqref{2.3a} for nonnegative $f\in \B_b(\H).$   Applying
 the log-Harnack inequality in Theorem \ref{T1.1} for
 $1+\vv f$ in place of $f$,  we obtain from the elementary inequality $r\leq \log (1 + r) + r^2 , \; r \geq 0,$

 $$P_t f(y)-\vv \|f\|_\infty^2 \le P_t \ff{\log (1+\vv f)}\vv (y)
   \le \ff 1 \vv \log (1+\vv P_t f(x)) + \ff{c_t \|x-y\|_0^2}\vv,\ \ \
 \vv>0, x,y\in \H,$$  where $c_t:= \ff{K}{2(1-\e^{-Kt})}.$
 Letting first $y\to x$ in $\|\cdot\|_0$ and then $\vv\to 0$, we obtain

 $$\limsup_{\|y-x\|_0\to 0} P_t f(y) \le P_t f(x).$$ Similarly, we have

 $$P_t\ff{\log (1+\vv f)}\vv (x) -\ff{c_t \|x-y\|_0^2}\vv
  \le \ff 1 \vv \log (1+\vv P_t f(y))  \le  P_t f(y).$$
  Letting first  $y\to x$ in $\|\cdot\|_0$ then $\vv\to 0$, we arrive at

  $$P_tf(x)\le \liminf_{\|x-y\|_0\to 0} P_t f(y).$$ Therefore, $P_tf$
  is $\|\cdot\|_0$ continuous. The second part of assertion \eqref{G2} is then
  an easy consequence.

  Now, let $p_t(x,y)$ be the transition density of $P_t$ with respect to $\mu$.
  By Theorem \ref{T1.1}, for any positive $f\in \B_b(\H)$ we have

 $$ \e^{P_t\log f(x)}\le \exp\bigg[\ff{K\|x-y\|_0^2}{2(1-\e^{-Kt})} \bigg]P_t f(y),\ \ \ x,y\in \H.$$
 Thus,

 $$\e^{P_t\log f(x)} \int_\H \exp\bigg[-\ff{K\|x-y\|_0^2}{2(1-\e^{-Kt})} \bigg]\mu(\d y)
 \le \int_\H P_t f(y)\mu(\d y)=C \e^{\alpha t}\mu(f).$$
    For
  fixed $x\in\H$,  applying this inequality  to $f= n\land p_t(x, \cdot)$
  then letting $n\to\infty$, we obtain

  $$\e^{\int_\H p_t(x,z)\log p_t (x,z)\mu(\d z)}\int_\H \exp\bigg[-\ff{K\|x-y\|_0^2}{2(1-\e^{-Kt})} \bigg]\mu(\d y)\le C \e^{\alpha t}.$$
  This implies (2).

  By approximations it remains to prove (3) for bounded
 positive $f$ with $\mu(f)=1.$ By Theorem \ref{T1.1} for $P_t^*f$ in place of $f$,
 we obtain

 $$P_t \log P_t^* f(x)\le \log P_tP_t^* f(y) + \ff{K\|x-y\|_0^2}{2(1-\e^{-Kt})},\ \ \ x,y\in \H.$$
 So, for any $\pi\in \C(f\mu,\mu)$, integrating both sides with respect to $\pi(\d x,\d y)$
 we arrive at

 $$\mu((P_t^* f)\log P_t^*)\le \mu(\log P_tP_t^*f) +\ff{K }{2(1-\e^{-Kt})}\int_{\H\times\H}
 \|x-y\|_0^2\pi(\d x,\d y).$$
Since  Jensen's inequality implies

$$ \mu(\log P_tP_t^*f) \le \log \mu(P_tP_t^*f)=\log\mu(f)=0,$$this implies the desired
entropy-cost inequality.

\end{document}